\documentclass[10pt]{article}

\usepackage[english]{babel}
\usepackage{amsmath,amsthm}
\usepackage{amsfonts}
\usepackage{graphicx}

\newtheorem{theorem}{Theorem}[section]

\newtheorem{lemma}{Lemma}[section]

\newtheorem{example}{Example}[section]

\def\qed{\qquad {$ \Box $} }

\def\qed{\qquad {$ \Box $} }

\begin{document}
\title{\bf  Calculating Entanglement Eigenvalues for Non-Symmetric Quantum Pure States Based on the Jacobian Semidefinite Programming Relaxation Method\thanks{The first and the third authors' work is partially supported by the Research Programme of National University of Defense Technology (No. ZK16-03-45), and the second author's work is partially supported by the National Science Foundation of China (No. 11471242).}}
\author{{Mengshi Zhang$^a$, Xinzhen Zhang$^b$, and Guyan Ni$^a$\thanks{Corresponding author. \newline {\it $\mathrm{\ \ \ \ }$
E-mail address}: guyan-ni@163.com(Guyan Ni); msh$_{-}$zhang@163.com(Mengshi Zhang); xzzhang@tju.edu.cn (Xinzhen Zhang).} %
 }\\
 {\small\it a College of Science, National University of
Defense Technology, Changsha, Hunan 410073, China.}\\
 {\small\it b\ Department of Mathematics, School of Science, Tianjin University, Tianjin, 300072, China.}\\
\\ \newline
\begin{tabular}{p{14.8cm}} \hline\\ {\small {\bf Abstract} \newline\quad
The geometric measure of entanglement is a widely used entanglement measure for quantum pure states. The key problem of computation of the geometric measure is to calculate the entanglement eigenvalue, which is equivalent to computing the largest unitary eigenvalue of a corresponding complex tensor. In this paper, we propose a Jacobian semidefinite programming relaxation method to calculate the largest unitary eigenvalue of a complex tensor. For this, we first introduce the Jacobian semidefinite programming  relaxation method for a polynomial optimization with equality constraint, and then convert the problem of computing the largest unitary eigenvalue to a real equality constrained polynomial optimization problem, which can be solved by the Jacobian semidefinite programming relaxation method. Numerical examples are presented to show the availability of this approach.
\medskip%
$~~$\newline{\it Keywords:} Jacobian semidefinite programming relaxation; entanglement eigenvalue; unitary eigenvalue; polynomial optimization; complex tensor} \\ \\
\hline
\end{tabular}
}
\date{}
\maketitle

\vskip 2mm

\section{Introduction}
\label{intro}

Quantum entanglement, first introduced by Einstein and Schir\"{o}dinger \cite{Ein35,Sch35}, has drawn much attention in the last decades. There are different measures of entanglement \cite{Ben96,Ve97,Har03} to quantify the minimum distance between a general state and the set of separable states, and the geometric measure is one of the most widely used measures for pure states, which was first proposed by Shimony \cite{S95} and generalized to multipartite system by Wei and Goldbart \cite{WG03}.

A key problem to computing the geometric measure of entanglement is to find the entanglement eigenvalue \cite{NQB14,Hill10}, which can be mathematically formulated as best rank-one approximation problem to a higher order tensor or a tensor eigenvalue computation problem \cite{Hay09,hqz12}. The Z-eigenvalue of a real tensor was first introduced by Qi \cite{qi05}. If the corresponding tensor of a quantum pure state is a real symmetric and non-negative tensor, it was shown that the entanglement eigenvalue is equal to the largest Z-eigenvalue \cite{hqz12}. However, Ni et al. \cite{NQB14} found that not all largest Z-eigenvalues of real tensors are the entanglement eigenvalue of pure states. To study geometric measure of entanglement by complex tensor analysis, Ni et al. \cite{NQB14} introduced the concept of unitary eigenvalue (U-eigenvalue) of complex tensors and showed that, for a pure state, its entanglement eigenvalue is equal to the largest U-eigenvalue and the nearest separable state is the corresponding unitary eigenvector (U-eigenvector).

It is shown that the problem of computing eigenvalues or the best rank-one approximation of a high order tensor is  NP-hard \cite{Hill13}. The existing methods mainly focus on real symmetric tensors \cite{QWW09,kdm11,HCD15,YYXSZ16} or complex symmetric tensors \cite{NB16}. Che et al. \cite{CCW17} proposed a neural networks method for computing  the best rank-one approximation of tensors. The Jacobian semidefinite relaxation method can also be used to compute the eigenpairs or the best rank-one problem of tensors. The method by Lasserre \cite{Las2001} was used to get the largest or smallest eigenvalue. Recently, Nie \cite{Nie2015} proposed a method for computing the hierarchy of local minimums in polynomial optimization, which uses the Jacobian semidefinite programming (SDP) relaxation method from \cite{Nie13}. Following the method in \cite{Nie13}, Cui et al. \cite{CuiDN2014}  computed all real eigenvalues sequentially. Nie and Wang \cite{NieW2014} used this method to solve the best rank-one tensor approximation problem. Hua et al. \cite{HNZ16} also used the approach to compute the geometric measure of entanglement for symmetric pure states. Most of these methods are  concentrated on the computation of the eigenpairs of symmetric tensors. However, there are few studies devoted to the non-symmetric case.

Motivated by above research and the relationship between quantum states and complex tensors, we discuss the computing of the entanglement eigenvalues of  non-symmetric pure states by the Jacobian SDP relaxation method. The remainder of this paper is organized as follows. In next section, we show some preliminaries about complex tensors, geometric measure of quantum entanglement, and their relationships. In Section 3, we first introduce the Jacobian SDP relaxation technique for equality constraint; Then, we use the method to compute the largest U-eigenvalues of non-symmetric complex tensors and their corresponding eigenvectors. In Section 4, numerical examples are carried out for different kinds of pure states.

\section{Preliminaries}
\subsection{Complex Tensors and U-Eigenvalues}

An $m$th-order complex tensor denoted by $\mathcal{A}=(\mathcal{A}_{i_1...i_m})\in H=\mathbb{C}^{n_1\times\cdots\times n_m}$ is a multiway array consisting of numbers $\mathcal{A}_{i_1...i_m}\in\mathbb{C}$ for all $i_k=1, 2, \cdots, n_k$, and $k=1, 2, \cdots, m$. A tensor $\mathcal{S}=(\mathcal{S}_{i_1...i_m})\in \mathbb{C}^{n\times\cdot\cdot\cdot\times n}$ is called symmetric, if its entries $\mathcal{S}_{i_1...i_m}$ are invariant under any permutation of $[i_1,...,i_m]$. For $\mathcal{A},\mathcal{B}\in H$, the inner product and norm are defined as
  $$
  \langle \mathcal{A},\mathcal{B}\rangle
  :=\sum^{n_1,...,n_m}_{i_1,...,i_m=1}\mathcal{A}^*_{i_1...i_m}\mathcal{B}_{i_1...i_m},\
  ||\mathcal{A}||_F:=\sqrt{\langle \mathcal{A},\mathcal{A}\rangle},
  $$
where $\mathcal{A}^*_{i_1...i_m}$ denotes the complex conjugate of $\mathcal{A}_{i_1...i_m}$.

Given $m$ vectors $z^{(i)}\in \mathbb{C}^{n_i},\ i=1,...,m $, a rank-one complex tensor $\otimes_{i=1}^m z^{(i)} $ is defined as
$$(\otimes_{i=1}^m z^{(i)})_{i_1...i_m}:=z^{(1)}_{i_1}\cdot \cdot \cdot z^{(m)}_{i_m}.$$

Define the inner product of a tensor and a rank-one tensor as follows:
$$
 \langle \mathcal{A},\otimes_{i=1}^m z^{(i)}\rangle
 :=\sum_{i_1,\cdots,i_m=1}^{n_1,\cdots,n_m} \mathcal{A}_{i_1\cdots i_m}^* z_{i_1}^{(1)}\cdots z_{i_m}^{(m)}.
$$

By the tensor product, $\langle \mathcal{A},\otimes_{i=1,i\not=k}^m z^{(i)}\rangle$ denotes a vector in $\mathbb{C}^{n_k}$, with $i_k$-th component being
$$
\langle \mathcal{A},\otimes_{i=1,i\not=k}^m z^{(i)}\rangle_{i_k} :=\sum_{i_1,\cdots,i_{k-1},i_{k+1},\cdots,i_m=1}^{n_1,\cdots,n_{k-1},n_{k+1},\cdots,n_m} \mathcal{A}_{i_1\cdots i_k\cdots  i_m}^* z_{i_1}^{(1)}\cdots z_{i_{k-1}}^{(k-1)}z_{i_{k+1}}^{(k+1)} \cdots z_{i_m}^{(m)}.
$$

A rank-one complex tensor $\otimes_{i=1}^m z^{(i)}$ is called {\it the best complex rank-one approximation} to $\mathcal{A}$, if it is a solution of the following optimization problem \cite{NQB14}:
 \begin{equation}\label{best appro1}
 \min_{z^{(i)}\in\mathbb{C}^{n_i},||z^{(i)}||=1}||\mathcal{A}-\otimes^m_{i=1} z^{(i)}||_F^2.
 \end{equation}

Assume that $n_1=n_2=\cdots=n_m=n$. Let $z=(z_1, \cdots, z_n)^T\in\mathbb{C}^n$. The rank-one tensor $\otimes_{i=1}^m z$ is called {\it a symmetric rank-one complex tensor}, abbreviated as $z^m$.

A symmetric rank-one complex tensor $\otimes_{i=1}^m z$ is called {\it the best symmetric complex rank-one approximation} to a symmetric tensor $\mathcal{S}$, if it is a solution of the following optimization problem \cite{NQB14}:
 \begin{equation}\label{best appro2}
  \min_{z\in\mathbb{C}^n,||z||=1}||\mathcal{S}-\otimes^m_{i=1} z||_F^2.
 \end{equation}

If $\mathcal{A}$ is symmetric, then optimization problem (\ref{best appro1}) is equivalent to (\ref{best appro2}) \cite{NQB14,HKWGG2009}. It means that the best symmetric complex rank-one approximation is also the best complex rank-one approximation for a symmetric complex tensor. In order to solve these optimization problems, Ni, Qi and Bai \cite{NQB14} defined U-eigenvalue and US-eigenvalue.

A real number $\lambda\in \mathbb{R}$ is called a {\it unitary eigenvalue (U-eigenvalue)} of $\mathcal{A}$, if $\lambda$ and a rank-one tensor $\otimes_{i=1}^m x^{(i)}$ is a solution pair of the following equation system:
\begin{equation}\label{EQ:Ueigen0}
\left\{
  \begin{array}{ll}
    \langle \mathcal{A},\otimes_{i=1,i\not=k}^m x^{(i)}\rangle=\lambda {x^{(k)}}^*, \\
    \lambda\in \mathbb{R}, x^{(i)}\in \mathbb{C}^{n_i}, ||x^{(i)}||=1,\ i=1,2,\cdots,m,
  \end{array}
\right.\ k=1,2,\cdots, m.
\end{equation}

For a symmetric tensor $\mathcal{S}$, a real number $\lambda\in \mathbb{R}$ is called a {\it unitary symmetric eigenvalue (US-eigenvalue)} of $\mathcal{S}$, if and vector $x$ solve the following equation system:
\begin{equation}\label{EQ:USeigen0}
\langle \mathcal{S},\otimes_{i=1}^{m-1} x\rangle=\lambda x^*,\ \lambda\in \mathbb{R},\ x\in \mathbb{C}^{n},\ ||x||=1.
\end{equation}

Furthermore, if $\mathcal{S}$ is real symmetric tensor, $\lambda$ is a real number and $x$ is a real vector, $\{\lambda, x \}$ solve the following equation system:
\begin{equation}\label{EQ:Zeigen0}
\langle \mathcal{S},\otimes_{i=1}^{m-1} x\rangle=\lambda x,\ \lambda\in \mathbb{R},\ x\in \mathbb{R}^{n},\ x^\top x=1,
\end{equation}
then $\lambda$ is called a {\it Z-eigenvalue} of $\mathcal{S}$.

Note that
\begin{equation}\label{EQ:distentsqure}
  ||\mathcal{A}-\otimes^m_{i=1} z^{(i)}||_F^2 = ||\mathcal{A}||_F^2+||\otimes^m_{i=1} z^{(i)}||_F^2-\langle \mathcal{A},\otimes^m_{i=1} z^{(i)}\rangle-\langle \otimes^m_{i=1} z^{(i)},\mathcal{A}\rangle .
\end{equation}
Hence, the minimization problem (\ref{best appro1}) is equivalent to the maximization problem:
\begin{equation}\label{max appro1}
   \max ~~\langle\mathcal{A},\otimes_{i=1}^m z^{i}\rangle+\langle\otimes_{i=1}^m z^{i},\mathcal{A}\rangle,\ s.t.~~||z^{(i)}||=1, ~~z^{(i)}\in\mathbb{C}^{n_i}.
\end{equation}
The critical point of the equivalent optimization problem (\ref{max appro1}) is given by
\begin{equation}\label{EQ:Ueigen1}
\left\{
  \begin{array}{ll}
    \langle \mathcal{A},\otimes_{i=1,i\not=k}^m z^{(i)}\rangle=\lambda {z^{(k)}}^*, \\
    \langle \otimes_{i=1,i\not=k}^m z^{(i)},\mathcal{A}\rangle=\lambda {z^{(k)}}, \\
    \lambda\in \mathbb{C}, ||z^{(i)}||=1,\ i=1,2,\cdots,m,
  \end{array}
\right.\ k=1,2,\cdots, m.
\end{equation}

Since $\langle \mathcal{A},\otimes_{i=1,i\not=k}^m z^{(i)}\rangle=\langle \otimes_{i=1,i\not=k}^m z^{(i)},\mathcal{A}\rangle^* $, (\ref{EQ:Ueigen0}) and (\ref{EQ:Ueigen1}) are the same. Following the fact that the largest absolute value of U-eigenvalue of the tensor $\mathcal{A}$ is the solution of the problem (\ref{max appro1}), the corresponding rank-one tensor $ \otimes_{i=1}^m z^{(i)}$ is the best rank-one approximation of $\mathcal{A}$.

\subsection{Multipartite Pure States and their Geometric Measure of Entanglement}
An $m$-partite pure state $|\psi\rangle$ of a composite quantum system can be regarded as a normalized element in a Hilbert space $H =\otimes_{k=1}^{m} H_k$, where
$H_k=\mathbb{C}^{n_k}$, $k = 1,2,$ $ \cdots ,m.$ Assume that $\{|e^{(k)}_{i_k}\rangle: i_k=1,2,\cdots, n_k\}$ is an orthonormal basis of $H_k$. Then,  $\{|e^{(1)}_{i_1} e^{(2)}_{i_2} \cdots e^{(m)}_{i_m}\rangle: i_k=1,2,\cdots, n_k;\ k=1,2,\cdots, m\}$ is an orthonormal basis of $H$. $|\psi\rangle$ is defined by
\begin{equation}\label{psi}
|\psi\rangle :=\sum_{i_1,\cdots,i_m=1}^{n_1,\cdots,n_m} x_{i_1\cdots i_m} |e^{(1)}_{i_1} e^{(2)}_{i_2} \cdots e^{(m)}_{i_m}\rangle,
\end{equation}
where $x_{i_1...i_m}\in\mathbb{C}$. $|\psi\rangle$ is called symmetric, if these amplitudes are invariant under permutations of the parties. A separable $m$-partite pure state is denoted as
$$|\phi\rangle :=\otimes_{k=1}^m |\phi^{(k)}\rangle, $$
where the index $k=1,\cdots, m$ labels the parts, and
$$
|\phi^{(k)}\rangle:=\sum_{i_k=1}^{n_k} x_{i_k}^{(k)} |e_{i_k}^{(k)}\rangle.
$$

Denote by $Separ(H)$ the set of all separable pure states $|\phi\rangle$ in $H$, subject to the constraint $\langle\phi|\phi\rangle=1$. The geometric measure of a given $m$-partite pure state $|\psi\rangle$ is defined as\cite{WG03}
\begin{equation}\label{geometric measure1}
E_G(|\psi\rangle):=\min_{|\phi\rangle\in Separ(H)} || |\psi\rangle-|\phi\rangle ||_F.
\end{equation}
Minimization problem (\ref{geometric measure1}) is equivalent to the following maximization problem:
\begin{equation}\label{entanglement eigenvalue}
G(|\psi\rangle):=\max_{|\phi\rangle\in Separ(H)} |\langle\psi|\phi\rangle|.
\end{equation}
The maximum of $G(|\psi\rangle)$ is called the \emph{entanglement eigenvalue}.

\subsection{The Relation of Multipartite Pure States and Complex Tensors }

For an $m$-partite pure state state $|\psi\rangle$ defined as in (\ref{psi}), the mutliway array consisting of $x_{i_1...i_m}$ can be denoted by a complex tensor $\mathcal{X}$. We call the tensor $\mathcal{X}$ as a corresponding tensor of $|\psi\rangle$ under an orthonormal basis of $H$. Hence, if an orthonormal basis of $H$ is given, then there is a 1-1 map between $m$-partite pure state states and $m$th-order complex tensors.

\begin{theorem}\label{TH:pmlambda}
Assume that $\mathcal{A}$ is an $m$th-order complex tensor. If $\lambda$ is a U-eigenvalue of $\mathcal{A}$, then $-\lambda$ is also a U-eigenvalue.
\end{theorem}

\begin{proof} Assume that $\eta=\sqrt[m]{-1}$. If $\lambda$ is a U-eigenvalue of $\mathcal{A}$ associated with rank-one tensor $ \otimes_{i=1}^m z^{(i)}$, then
$$
    \langle \mathcal{A},\otimes_{i=1,i\not= k}^{m}(\eta z^{(i)})\rangle=-\lambda (\eta z^{(k)})^*,\ k=1, \cdots, m.
$$
It follows that $-\lambda$ is also a U-eigenvalue. This completes the proof. \qed
\end{proof}

\begin{theorem}\label{TH:Glambdamax} Assume that $\mathcal{X}$ is the corresponding tensor of a multipartite pure state $|\psi\rangle$ under a orthonormail basis as in {\rm (\ref{psi})}. Let $\lambda_{max}$ be the largest U-eigenvalue of $\mathcal{X}$. Then,

(a) $G(|\psi\rangle)=\lambda_{max}$,

(b) $ E_G(|\psi\rangle)=\sqrt{2-2\lambda_{max}}.$
\end{theorem}

\begin{proof} (a) Assume that $\lambda_{max}$ is the largest U-eigenvalue of $\mathcal{X}$ with a corresponding rank-one tensor $ \otimes_{i=1}^m z^{(i)}$. Let $|\phi\rangle =\otimes_{k=1}^m |\phi^{(k)}\rangle $ where $|\phi^{(k)}\rangle=\sum_{i=1}^{n_k} z_{i}^{(k)}|e_{i}^{(k)}\rangle$ for all $k=1,2,\cdots,m$. By Theorem \ref{TH:pmlambda}, it is known that if $\lambda$ is a U-eigenvalue of $\mathcal{A}$, then $-\lambda$ is also a U-eigenvalue. Hence, the maximal absolute value of U-eigenvalues is also a U-eigenvalue of $\mathcal{A}$. Then we have
$$
\lambda_{max}=\langle \psi | \phi \rangle =\max_{|\varphi\rangle\in Separ(H)}|\langle \psi | \varphi \rangle|.
$$
It follows that $G(|\psi\rangle)=\lambda_{max}$.

(b) The second result follows directly from (\ref{EQ:distentsqure}) and (\ref{geometric measure1}).  \qed
\end{proof}

\section{Jacobian SDP Relaxation Method for Calculating Geometric Measure of Entanglement for Pure States}
In this section, we introduce a Jacobian SDP relaxation method and a polynomial optimization method to compute geometric measure of entanglement for a non-symmetric pure state. The key point of computing geometric measure of entanglement is to compute the largest U-eigenvalue of non-symmetric complex tensors. There are many literatures focus on computing eigenvalues or the largest eigenvalue of a symmetric real tensor. To proceed, we first introduce a Jacobian SDP relaxation method for equality constraint, and then introduce a polynomial optimization method to compute the largest U-eigenvalue of a non-symmetric complex tensor via Jacobian SDP relaxation method.

\subsection{The Jacobian SDP Relaxation Method for Equality Constraint}
Consider a real-valued polynomial optimization problem
\begin{equation}\label{Eq:ployoptim}
  \left\{
  \begin{array}{ll}
   \min_{x\in \mathbb{R}^n} ~~f(x) \\
   s.t.~~h_i(x)=0,\ i=1,2,...,r_1, \\
   ~~~~~~g_j(x)\geq 0,\ j=1,2,...,r_2,
  \end{array}
\right.
\end{equation}
where $f(x), h_i(x), g_j(x)$ are polynomial functions in $ x\in \mathbb{R}^n$.
Let $f_{min}$ be its global minimum. The problem of finding $f_{min}$ is NP-hard \cite{Nie2015}. A standard approach for solving (\ref{Eq:ployoptim}) is SDP relaxations proposed by Lasserre \cite{Las2001}. It is based on a sequence of sum of squares type representations of polynomials that are non-negative on its feasible set.

A new SDP type relaxation for solving (\ref{Eq:ployoptim}) was proposed by Nie \cite{Nie13}, and the involved polynomials are only in $x$. Suppose the feasible set is non-singular and $f_{min}$ is achievable, which is true generically. Nie constructed a set of new polynomials, $\varphi_1(x),\ \cdots, \varphi_r(x)$, by using the minors of the Jacobian of $f, h_i, g_j$, such that (\ref{Eq:ployoptim}) is equivalent to
\begin{equation}\label{Eq:ployoptim2}
  \left\{
  \begin{array}{ll}
   \min_{x\in \mathbb{R}^n} ~~f(x) \\
   s.t.~~h_i(x)=\varphi_j(x)=0,\ i=1,\cdots,r_1, j=1,\cdots,r, \\
   ~~~~~~\prod_{k=1}^{r_2} g_k(x)^{v_k}\geq 0,\ v_k\in \{0,1\},\ k=1,\cdots, r_2.
  \end{array}
\right.
\end{equation}
Nie \cite{Nie13} proved that, for all $N$ big enough, the standard $N$-th order Lasserre¡¯s relaxation for the above returns a lower bound that is equal to the minimum $f_{min}$. That is, an exact SDP
relaxation for (\ref{Eq:ployoptim}) is obtained by the Jacobian SDP relaxation method. Recently, Hua et,al. \cite{HNZ16} convert the problem of computing the geometric measure of entanglement for symmetric pure states to a real polynomial optimization problem and solve it through the Jacobian SDP relaxation method.

Here, we briefly review the equality constrained Jacobian SDP relaxation problem. Let $u\in \mathbb{R}^{2n}$, $f(u)$ be a real homogeneous polynomial function on $u$ with degree $m$, and $g(u)$ be a real polynomial function. Consider the following optimization.
\begin{equation}\label{convert(14)}
   \max ~~f(u) \quad s.t.~~g(u)=0.
\end{equation}

\begin{lemma}{\rm \cite{HNZ16}}
  The polynomial optimization problem {\rm (\ref{convert(14)})} is equivalent to
\begin{equation}\label{eq:polyopt2}
   \max ~~f(u) \quad s.t.~~g(u)=0,h_r(u)=0,~1\leq r\leq 4n-3,
\end{equation}
where $u\in \mathbb{R}^{2n}$ and
\begin{equation}\label{Eq:Hr}
 h_r:=\sum_{i+j=r+2}(f'_{u_i}g'_{u_j}-f'_{u_j}g'_{u_i})=0,~1\leq r\leq 4n-3,
\end{equation}
$f'_{u_i}$ and $g'_{u_i}$ denote partial derivatives of $f$ and $g$ for $u_i$, respectively.
\end{lemma}

Let $q(u)$ be a polynomial with $deg(q)\leq2N$. Assume that $y$ is a moment vector indexed by $\alpha\in \mathbb{N}^{2n}$ with $|\alpha|\leq 2N$. The N-th order localizing matrix of $q$ is defined as
$$
L^{(N)}_q(y) :=\sum_{\alpha\in \mathbb{N}^{2n}:|\alpha|\leq 2N} A^{(N)}_{\alpha}y_{\alpha},
$$
where the symmetric matrices $A^{(N)}_{\alpha}$ satisfy:
$$
q(u)[u]_d[u]_d^T=\sum_{\alpha\in \mathbb{N}^{2n}:|\alpha|\leq 2N} A^{(N)}_{\alpha}u^{\alpha},
$$
$d=N-\lceil deg(q)/2 \rceil$ and $[x]_d$ is the monomial vector:
$$
[u]_d:=[1, u_1, u_2, \cdots, u_{2n}, u_1^2, u_1u_2, \cdots, u_{2n}^2, \cdots, u_1^d,\ u_1^{d-1}u_2,\ \cdots,\ u_{2n}^d]^T.
$$

When $q=1$, the $L^{(N)}_1(y)$ is denoted by $M_N(y)$.
Lasserre's SDP relaxations for solving (\ref{eq:polyopt2}) is
\begin{equation}\label{SDP}
  \left\{
    \begin{array}{ll}
      \rho_N:=& \max \limits_{\alpha\in \mathbb{N}^{2n}:|\alpha|=m}f_{\alpha}y_{\alpha} \\
      s.t & L_{g}^{(N)}(y)=0,~L_{h_r}^{(N)}(y)=0,\ (r=1,...,4n-3)\\
      ~~~ & y_0=1, M_N(y)\succeq 0.
    \end{array}
  \right.
\end{equation}

Denote $f_{max}$ as the maximum of (\ref{eq:polyopt2}), the sequences $\{\rho_N\}$ is monotonicallly decreasing and is upper bounds for $f_{max}$.

\begin{theorem}{\rm \cite{HNZ16}}
  When Lasserre's hierarchy of semidefinite relaxations is applied to solve {\rm (\ref{eq:polyopt2})}, for all $N$ big enough, the standard $N$-th order Lasserre's relaxation for {\rm (\ref{eq:polyopt2})} returns the maximum $f_{max}$.
\end{theorem}

\subsection{The Jacobian SDP Relaxation Method for the Largest U-Eigenvalue of a Non-Symmetric Tensor}
Let $\mathcal{A}$ be a non-symmetric complex tensor. The maximum optimization problem (\ref{max appro1}) is equivalent to the following optimization
problem:

\begin{equation}\label{cunitsphere}
\begin{array}{rl}
\hat{f}:=\max&\mbox{Re} (\langle\mathcal{A}, \otimes_{i=1}^{m} z^{(i)}\rangle)\\
{s.t.} & \|z^{(i)}\|=1, ~~z^{(i)}\in \mathbb{C}^{n_i},~~i=1,\ldots,m
\end{array}
\end{equation}
Clearly, the largest U-eigenvalue $\lambda=\hat{f}$ and the optimal solution is the corresponding  U-eigenvector of $\mathcal{A}$.
Note that any complex number $c$
can be expressed as two real numbers $a,b$ with $c=a+\sqrt{-1}b$. So (\ref{cunitsphere}) can be rewritten as
 \begin{equation}\label{runitsphere}
\begin{array}{rl}
\hat{f}:=\max&\mbox{Re} (\langle\mathcal{A}, \otimes_{i=1}^{m}(x^{(i)}+\sqrt{-1} y^{(i)}\rangle)\\
{s.t.} & ||x^{(i)}||^2+||y^{(i)}||^2=1,
x^{(i)}, y^{(i)}\in \mathbb{R}^{n_i}, ~~i=1,\ldots ,m
\end{array}
\end{equation}
Note that the objective function of (\ref{runitsphere}) is a real multilinear function of degree $m$ and dimension $2(n_1+\dots+n_m)$.
For convenience, let tensor $\mathcal{B}\in \mathbb{R}^{2n_1\times \dots\times 2n_m}$ satisfy
\begin{equation}\label{complex2realT}
  \langle\mathcal{B}, \otimes_{i=1}^{m} u^{(i)}\rangle=\mbox{Re} (\langle\mathcal{A}, \otimes_{i=1}^{m}(x^{(i)}+\sqrt{-1} y^{(i)}\rangle)
\end{equation}
with $u^{(i)}=(x^{(i)}, y^{(i)}) \in \mathbb{R}^{2n_i}$ for $i=1,\dots, m.$
From the previous discussion,  problem (\ref{runitsphere}) is equivalent to the following optimization
 \begin{equation}\label{psym}
\begin{array}{rl}
\hat{f}:=\max& \langle\mathcal{B}, u^{(1)} \otimes u^{(2)}\otimes\dots \otimes u^{(m)}\rangle\\
{s.t.} & \|u^{(i)}\|=1,
u^{(i)}\in \mathbb{R}^{2n_i}, ~~i=1,2\dots, m.
\end{array}
\end{equation}

Since problem (\ref{psym}) is a spherical multilinear optimization, to lower the dimension, (\ref{psym}) can be rewritten as
 \begin{equation}\label{ldop}
\begin{array}{rl}
\hat{f}:=\max& \|\langle\mathcal{B}, u^{(1)}\otimes u^{(2)}\otimes  \dots \otimes u^{(m-1)}\rangle\|\\
{s.t.} & \|u^{(i)}\|=1,
u^{(i)}\in \mathbb{R}^{2n_i}, ~~i=1,2\dots, m-1.
\end{array}
\end{equation}

Let $f(u)=\|\langle\mathcal{B}, u^{(1)}\otimes u^{(2)}\otimes  \dots \otimes u^{(m-1)}\rangle\|^2$, $g_k(u)=||u^{(k)}||^2-1$. Similar to (\ref{eq:polyopt2}), the optimization problem (\ref{ldop}) can be written as
 \begin{equation}\label{Eq:sdpnons}
\begin{array}{rl}
\hat{f}^{2}:=\max& f(u)\\
{s.t.} & g_k(u)=0, h_{k,r}=0 ~~k=1,2\dots, m-1,\ 1\leq r\leq 4n_k-3.
\end{array}
\end{equation}
where
\begin{equation}\label{Eq:Hir}
 h_{k,r}:=\sum_{i+j=r+2}((g_k)'_{u_j}f'_{u_i}-(g_k)'_{u_i}f'_{u_j}),~1\leq r\leq 4n_k-3,
\end{equation}

Lasserre's SDP relaxations for solving (\ref{Eq:sdpnons}) is
\begin{equation}\label{LasSDP1}
  \left\{
    \begin{array}{ll}
      \rho_N:=& \max \limits_{\alpha\in \mathbb{N}^{2n}:|\alpha|=2m-2}f_{\alpha}y_{\alpha} \\
      s.t & L_{g_k}^{(N)}(y)=0,~L_{h_{k,r}}^{(N)}(y)=0\ (k=1, \cdots, m-1, r=1,...,4n_k-3)\\
      ~~~ & y_0=1, M_N(y)\succeq 0.
    \end{array}
  \right.
\end{equation}


\begin{theorem}\label{Th:eigenpairs}
Let $x^{(i)}, y^{(i)}\in \mathbb{R}^{n_i}$, $u^{(i)}=(x^{(i)}, y^{(i)})\in \mathbb{R}^{2n_i}$ for $i=1,\dots, m.$ Let $\mathcal{A}$ be an $m$th-order complex tensor in $\mathbb{C}^{n_1\times\cdots\times n_m}$ and $\mathcal{B}$ be an $m$th-order real tensor in $\mathbb{R}^{2n_1\times\cdot\cdot\cdot\times 2n_m}$ satisfying {\rm (\ref{complex2realT})}.
Assume that $\{\hat{u}^{(i)}| i=1,\cdots, m-1\}$ is a maximizer of the optimization problem {\rm (\ref{ldop})} and $\lambda$ is the maximal value. Let
\begin{equation}\label{Eq:eigenUm}
 \hat{u}^{(m)}=\frac{\langle \mathcal{B}, \hat{u}^{(1)}\otimes\cdots\otimes\hat{u}^{(m-1)} \rangle}{\lambda},\ \hat{z}^{(i)}=\hat{x}^{(i)} + \hat{y}^{(i)}\sqrt{-1},\ i=1, \cdots, m.
\end{equation}
Then, $\lambda$ is the largest U-eigenvalue of $\mathcal{A}$ and $\{\hat{z}^{(1)},\cdots , \hat{z}^{(m)}\}$  is a tuple of corresponding U-eigenvector.
\end{theorem}

\begin{proof} By the assumption, $\{\hat{u}^{(i)}| i=1,\cdots, m-1\}$ is a maximizer of the optimization problem (\ref{ldop}) and $\lambda$ is the maximal value. Since again $\hat{u}^{(m)}$ is defined as in (\ref{Eq:eigenUm}). Hence, $\{\hat{u}^{(i)}| i=1,\cdots, m \}$ is a maximizer and $\lambda$ is the maximal value of the optimization problem (\ref{psym}). It follows that $\{\hat{z}^{(i)}| i=1,\cdots, m \}$ is a maximizer and $\lambda$ is the maximal value of the optimization problem (\ref{cunitsphere}), which means that
\begin{equation}\label{eq:maxlambda}
 \lambda=\frac{\langle\mathcal{A}, \otimes_{i=1}^{m} \hat{z}^{(i)}\rangle + \langle\otimes_{i=1}^{m} \hat{z}^{(i)}, \mathcal{A}\rangle}{2}.
\end{equation}
Since $\lambda$ is the maximal value, (\ref{eq:maxlambda}) implies that
$$
\langle\mathcal{A}, \otimes_{i=1,i\not=k}^{m} \hat{z}^{(i)}\rangle = \lambda \hat{z}^{(k)*},\ \mathrm{for\ } k=1, \cdots, m.
$$
It follows that $\lambda$ is the largest U-eigenvalue of $\mathcal{A}$ and $\{\hat{z}^{(1)},\cdots , \hat{z}^{(m)}\}$  is a tuple of corresponding U-eigenvector. This completes the proof.   \qed
\end{proof}

Next, we consider the case that tensor $\mathcal{A}$ is partially symmetric. Without loss of generality, we assume that the tensor is partially symmetric in the first two indices. That is, $\mathcal{A}_{i_1i_2i_3\dots i_m}=\mathcal{A}_{i_2i_1i_3\dots i_m}$ for any fixed $i_3,\dots, i_m$. It is clear that $\mathcal{B}$ is also partially symmetric in the first two indices.

Similarly to the proof of Theorem 2.1 in \cite{ZLQ12}, problem (\ref{runitsphere}) is equivalent to the following optimization
 \begin{equation}\label{ppsym}
\begin{array}{rl}
\hat{f}:=\max& \langle\mathcal{B}, (u^{(1)})^2 \otimes u^{(3)}\dots \otimes u^{(m)}\rangle\\
{s.t.} & \|u^{(i)}\|=1,
u^{(i)}\in \mathbb{R}^{2n_i}, ~~i=1,3\dots, m.
\end{array}
\end{equation}

Assume that a tuple of unit vectors $\{x^{(1)}, \dots, x^{(m)}\}$ is a solution of (\ref{ppsym}), then $$\langle\mathcal{B}, x^{(1)} \otimes x^{(2)}\otimes\dots \otimes x^{(m)}\rangle=|| \langle\mathcal{B}, x^{(1)} \otimes x^{(2)}\otimes\dots \otimes x^{(m-1)}\rangle ||$$ and $$ x^{(m)} = \langle\mathcal{B}, x^{(1)} \otimes x^{(2)}\otimes\dots \otimes x^{(m-1)}\rangle / \langle\mathcal{B}, x^{(1)} \otimes x^{(2)}\otimes\dots \otimes x^{(m)}\rangle. $$
It follows that
$$\max \langle\mathcal{B}, u^{(1)} \otimes u^{(2)}\dots \otimes u^{(m)}\rangle = \max \|\langle\mathcal{B}, u^{(1)}\otimes u^{(2)}\otimes  \dots \otimes u^{(m-1)}\rangle\|.$$

Without loss of generality,  we assume that $n_1\leq \dots\leq n_m$.  Since problem (\ref{ppsym}) is a spherical multilinear optimization, to lower the dimension, (\ref{ppsym}) can be rewritten as
 \begin{equation}\label{Eq:pldop}
\begin{array}{rl}
\hat{f}:=\max& \|\langle\mathcal{B}, (u^{(1)})^2\otimes u^{(3)}\otimes  \dots \otimes u^{(m-1)}\rangle\|\\
{s.t.} & \|u^{(i)}\|=1,
u^{(i)}\in \mathbb{R}^{2n_i}, ~~i=1,3, \cdots, m-1.
\end{array}
\end{equation}

Let $f(u)=\|\langle\mathcal{B}, (u^{(1)})^2\otimes u^{(3)}\otimes \dots \otimes u^{(m-1)}\rangle\|^2$ and $g_k(u)=||u^{(k)}||^2-1$. Then, similar to (\ref{eq:polyopt2}), optimization problem (\ref{Eq:pldop}) can be written as
 \begin{equation}\label{Eq:psdpnons}
\begin{array}{rl}
\hat{f}^{2}:=\max& f(u)\\
{s.t.} & g_k(u)=0, h_{k,r}=0 ~~k=1,3,4, \cdots, m-1,\ 1\leq r\leq 4n_k-3.
\end{array}
\end{equation}
where
\begin{equation}\label{Eq:Hir}
 h_{k,r}:=\sum_{i+j=r+2}((g_k)'_{u_j}f'_{u_i}-(g_k)'_{u_i}f'_{u_j}),~1\leq r\leq 4n_k-3,
\end{equation}

Lasserre's SDP relaxations for solving (\ref{Eq:psdpnons}) is
\begin{equation}\label{LasSDP2}
  \left\{
    \begin{array}{ll}
      \rho_N:=& \max \limits_{\alpha\in \mathbb{N}^{2n}:|\alpha|=2m-2}f_{\alpha}y_{\alpha} \\
      s.t & L_{g_k}^{(N)}(y)=0,~L_{h_{k,r}}^{(N)}(y)=0(k=1, 3, 4, \cdots, m-1, r=1,...,4n_k-3)\\
      ~~~ & y_0=1, M_N(y)\succeq 0.
    \end{array}
  \right.
\end{equation}

 Then, by the discussion of totally non-symmetric tensor, we use Jacobian SDP relaxation method to solve (\ref{Eq:psdpnons}) and get the largest U-eigenvalue $\lambda=\hat{f}$ and the $m$ corresponding U-eigenvectors. Because in this case the tensor is partially symmetric, the dimension of (\ref{Eq:pldop}) is much lower than that of (\ref{ldop}), which can help us to increase the computational efficiency.

\section{Numerical Examples}
Theorem \ref{TH:Glambdamax} illustrates that the entanglement eigenvalue of a quantum state $|\phi\rangle$ can be obtained by computing the largest U-eigenvalue of the corresponding complex tensor. For non-symmetric case, a polynomial optimization method can be used to compute the largest U-eigenvalue of a non-symmetric tensor. In this section, we present numerical examples of using the polynomial optimization method to find the largest U-eigenvalue of non-symmetric tensors.

The computations are implemented in MATLAB 2014a on a Microsoft Win10 Laptop with 8GB memory and Intel(R) CPU 2.40GHZ. We use the toolbox \emph{Gloptipoly 3} and \emph{SDPNAL+} to solve the SDP relaxation problems.
\begin{example}\label{EX: partically symmetric tensor}
Consider a non-symmetric 3-partite state
$$|\psi\rangle= \frac{1}{2}|000\rangle+ \frac{\sqrt{3}}{6}(|110\rangle+ |011\rangle+ |101\rangle)+(\frac{1}{2}+\frac{1}{2}\sqrt{-1})|001\rangle,$$
which corresponds a $3$-rd order $2\times2\times2$ complex tensor $\mathcal{A}$, and its non-zero entries are $\mathcal{A}_{111}=\frac{1}{2},\ \mathcal{A}_{221} =\mathcal{A}_{212}= \mathcal{A}_{122}=\frac{\sqrt{3}}{6},\  \mathcal{A}_{112}=\frac{1}{2}+\frac{1}{2}\sqrt{-1}$.
It is clear that $\mathcal{A}$ is a partially symmetric tensor in the first two indices. By (\ref{Eq:pldop}), the largest U-eigenvalue problem  is equivalent to solve the following optimization
 \begin{equation}\label{ex:Bu2}
\begin{array}{rl}
\hat{f}:=\max \|\langle\mathcal{B}, (u^{(1)})^2\rangle\| \quad {s.t.}  \|u^{(1)}\|=1, \quad u^{(1)}\in \mathbb{R}^{4},
\end{array}
\end{equation}
where $u^{(1)}=(u_1,u_2,u_3,u_4)^\top$, and
$$\|\langle\mathcal{B}, (u^{(1)})^2\rangle\|^2=\frac{u_1^2u_2^2}{3} + \frac{3u_1^2u_3^2}{2} + \frac{u_1^2u_4^2}{3} + \frac{u_2^2u_3^2}{3} +\cdots + \frac{2\sqrt{3}u_1u_2u_3u_4}{3}.
$$

We first convert the optimization problem (\ref{ex:Bu2}) to (\ref{Eq:psdpnons}) and set $u_3=0$ to avoid the situation of infinite number of solutions, and then we solve the polynomial optimization (\ref{Eq:psdpnons}). We obtain two maximizers
$$ \hat{u}^{(1)} =( -0.9625, -0.2242, 0, 0.1530)^\top \ \mathrm{or\ }
 \hat{u}^{(1)} =( 0.9625, 0.2242, 0, -0.1530)^\top. $$
and the maximal value $\hat{f}=0.9317$.

Hence, we get the largest U-eigenvalue $ \lambda=0.9317 $ and two tuples of corresponding U-eigenvectors
\begin{align*}
 \hat{z}^{(1)}&=\hat{z}^{(2)}=(-0.9625,  -0.2242 + 0.1530\sqrt{-1})^\top,\\
  \hat{z}^{(3)}&=(0.5054 + 0.0213\sqrt{-1}, 0.6308 + 0.5883\sqrt{-1})^\top
\end{align*}
 or
\begin{align*}
 \hat{z}^{(1)}&=\hat{z}^{(2)}=(0.9625 , 0.2242 - 0.1530\sqrt{-1})^\top,\\
  \hat{z}^{(3)}&=(0.5054 + 0.0213\sqrt{-1}, 0.6308 + 0.5883\sqrt{-1})^\top
\end{align*}
By Theorem \ref{TH:Glambdamax} we obtain the entanglement eigenvalue $G(|\psi\rangle)=\lambda=0.9317$.  \end{example}


\begin{example}\label{EX:totally non-symmetric tensor of 2*2*2}
  Consider a non-symmetric 3-partite state
$$|\psi\rangle= \frac{1}{6}|000\rangle+\frac{2}{3}\sqrt{-1}|111\rangle+ (\sqrt{\frac{1}{3}}+\frac{1}{3}\sqrt{-1})|101\rangle+ \frac{\sqrt{3}}{6}|100\rangle,$$
which corresponds a $3$rd-order $2$-dimension non-symmetric square tensor $\mathcal{A}$ with $\mathcal{A}_{111}=\frac{1}{6},$ $\mathcal{A}_{222}=\frac{2}{3}\sqrt{-1},$ $\mathcal{A}_{212}=(\sqrt{\frac{1}{3}}+\frac{1}{3}\sqrt{-1}),$ $\mathcal{A}_{211}=\frac{\sqrt{3}}{6}$.
It is easy to see that $\mathcal{A}$ is a totally non-symmetric tensor. According to the general case, the largest U-eigenvalue problem  is equivalent to solve the following optimization
 \begin{equation}\label{ex:Bu1u2}
\begin{array}{rl}
\hat{f}:=\max& \|\langle\mathcal{B}, u^{(1)}\otimes u^{(2)}\rangle\|\\
{s.t.} & \|u^{(i)}\|=1,
u^{(i)}\in \mathbb{R}^{4}, i=1,2.
\end{array}
\end{equation}
Here $u^{(1)}=(u^{(1)}_1,u^{(1)}_2,u^{(1)}_3 , u^{(1)}_4)^\top, u^{(2)}=(u^{(2)}_1,u^{(2)}_2,u^{(2)}_3,u^{(2)}_4)^\top$, and
$$||\langle\mathcal{B}, u^{(1)}\otimes u^{(2)}\rangle||^2=\frac{({u^{(1)}_1})^2({u^{(2)}_1})^2}{36} + \frac{({u^{(1)}_1})^2({u^{(2)}_3})^2}{36}+ \cdots  + \frac{4({u^{(1)}_4})^2({u^{(2)}_4})^2}{9}.
$$

We convert the optimization problem (\ref{ex:Bu1u2}) to (\ref{Eq:sdpnons}) and set $u^{(1)}_3=u^{(2)}_3=0$ to avoid the situation of infinite number of solutions. Then, we solve the polynomial optimization (\ref{Eq:sdpnons}). We get the largest U-eigenvalue $ \lambda=0.9661 $ and a tuple of corresponding U-eigenvectors as
\begin{eqnarray*}
  && \hat{z}^{(1)}=(-0.0287, -0.9996)^\top,\   \hat{z}^{(2)}=(-0.7404, -0.3361-0.5821\sqrt{-1})^\top, \\
  && \hat{z}^{(3)}= (0.2248, 0.8439+0.4872\sqrt{-1})^\top.
\end{eqnarray*}
By Theorem \ref{TH:Glambdamax}, we obtain the entanglement eigenvalue $G(|\psi\rangle)=\lambda=0.9661$.
 \end{example}



 \begin{example}\label{EX: Tensor with parameter 2*2*2}
  Consider a non-symmetric 3-partite state with parameter as the following
   $$|\psi\rangle= \sum_{i_1, i_2, i_3=1}^n \frac{\cos(i_1-i_2+i_3)+\sqrt{-1}\  \sin(i_1+i_2-i_3)}{\sqrt{n^3}}|(i_1-1)(i_2-1)(i_3-1)\rangle,$$
 $ |\psi\rangle $ corresponds to a $3$rd-order non-symmetric tensor $\mathcal{A}\in \mathbb{C}^{n\times n\times n}$ with $$\mathcal{A}_{i_1i_2i_3}= \frac{\cos(i_1-i_2+i_3)+\sqrt{-1}\  \sin(i_1+i_2-i_3)}{\sqrt{n^3}}.$$

For $n=2$, the largest U-eigenvalue is equivalent to the following optimization problem
 \begin{equation}
\max \|\langle\mathcal{B}, u^{(1)}\otimes u^{(2)}\rangle\|, \
{s.t.\ }  \|u^{(i)}\|=1,
u^{(i)}\in \mathbb{R}^{4}, i=1,2.
\end{equation}
Here $u^{(1)}=(u^{(1)}_1,u^{(1)}_2,u^{(1)}_3 , u^{(1)}_4)^\top, u^{(2)}=(u^{(2)}_1,u^{(2)}_2,u^{(2)}_3,u^{(2)}_4)^\top$.

We convert the optimization problem to (\ref{Eq:sdpnons}) and set $u^{(1)}_3=u^{(2)}_3=0$ to avoid the situation of infinite number of solutions. Then, we solve the polynomial optimization (\ref{Eq:sdpnons}). We get the largest U-eigenvalue $ \lambda=0.8895 $ and a tuple of corresponding U-eigenvectors as
\begin{eqnarray*}
  && \hat{z}^{(1)}=-(0.6928, 0.6734 + 0.2580 \sqrt{-1})^\top, \ \hat{z}^{(2)}=-(0.689, 0.450 - 0.5681\sqrt{-1})^\top, \\
  && \hat{z}^{(3)}=(0.1533 + 0.7083\sqrt{-1}, -0.4375 + 0.5324\sqrt{-1})^\top.
\end{eqnarray*}
By Theorem \ref{TH:Glambdamax}, we obtain the entanglement eigenvalue $G(|\psi\rangle)=\lambda=0.8895$. \end{example}


 \begin{example}\label{EX: Tensor with parameter 2*2*2}
  Consider a symmetric 3-partite state with parameter given as $$|\psi\rangle= \sum_{i_1, i_2, i_3=1}^n \frac{\cos(i_1+i_2+i_3)+\sqrt{-1}\  \sin(i_1+i_2+i_3)}{\sqrt{n^3}}|(i_1-1)(i_2-1)(i_3-1)\rangle,$$
 $ |\psi\rangle $ corresponds to a $3$rd-order symmetric tensor $\mathcal{A}\in \mathbb{C}^{n\times n\times n}$ with $$\mathcal{A}_{i_1i_2i_3}= \frac{\cos(i_1+i_2+i_3)+\sqrt{-1}\  \sin(i_1+i_2+i_3)}{\sqrt{n^3}}.$$

For $n=2$, the largest U-eigenvalue problem is equivalent to solve the following optimization
\begin{equation}
\begin{array}{rl}
\hat{f}:=\max& \mathrm{Re} \langle\mathcal{A}, z^3\rangle,\ 
{\mathrm{s.t.}}\ \|z\|=1,
z\in \mathbb{C}^{2}.
\end{array}
\end{equation}
We solve the polynomial optimization by SDP relaxation method and obtain the largest U-eigenvalue $ \lambda=1 $ and a tuple of corresponding U-eigenvector $$\hat{z}=(0.382051 + 0.59501 \sqrt{-1}, -0.29426 + 0.64297 \sqrt{-1})^\top.$$ It is easy to verify that $\mathcal{A}=\hat{z}\otimes\hat{z}\otimes\hat{z}$. Hence, this is a rank-one tensor, which means that the pure state $|\psi\rangle$ is a separable state. Let
$$|\phi\rangle=(0.382051 + 0.59501 \sqrt{-1})|0\rangle + (-0.29426 + 0.64297 \sqrt{-1})|1\rangle.$$
Then, $|\psi\rangle=|\phi\rangle\otimes|\phi\rangle\otimes|\phi\rangle$.

\end{example}

 \end{document}